\RequirePackage{fix-cm}
\documentclass[smallextended]{svjour3}     
\usepackage{graphicx}
\usepackage{amsfonts,amsmath,bm,dsfont,mathrsfs}
\usepackage{amssymb,pifont}
\usepackage{enumitem,cases}
\usepackage{color,calc}
\usepackage{hyperref}[6.83]
\hypersetup{
  colorlinks = true,
  allcolors = refblue,
  urlcolor = refblue,
}
\definecolor{refblue}{RGB}{26,13,171}

\numberwithin{equation}{section}
\numberwithin{lemma}{section}
\numberwithin{definition}{section}
\numberwithin{assumption}{section}
\numberwithin{theorem}{section}
\numberwithin{proposition}{section}
\numberwithin{corollary}{section}
\numberwithin{remark}{section}

\def\cA{{\mathcal A}}
\def\cB{{\mathcal B}}

\def\cE{{\mathcal E}}
\def\cF{{\mathcal F}}

\def\cO{{\mathcal O}}

\def\cV{{\mathcal V}}

\newcommand{\cone}{{\sf cone}}
\renewcommand{\ker}{{\sf ker}}
\newcommand{\range}{{\sf rge}}
\newcommand{\trace}{{\sf trace}}
\renewcommand\Re{{\mathbb R}}
\newcommand{\boundary}{{\sf bdry}}
\newcommand{\conv}{{\sf conv}}

\title{On the $p$-order Semismoothness of the Metric Projection onto Slices of the Positive Semidefinite Cone \thanks{{This work was funded in part by the Hong Kong Research Grants Council under GRF project 15305324 and the RGC Senior Research Fellow scheme SRFS2223-5S02.}}
}

\author{
Ruoning Chen \and 
Jiaming Ma \and
Defeng Sun
}

\institute{
Ruoning Chen
\at
Department of Mathematical Sciences, Tsinghua University, Beijing, China; and 
Department of Applied Mathematics, The Hong Kong Polytechnic University, Hung Hom, Hong Kong
\\
\email{crn22@mails.tsinghua.edu.cn; ruoning.chen@polyu.edu.hk} 
\and
Jiaming Ma
\at
Department of Applied Mathematics, The Hong Kong Polytechnic University, Hung Hom, Hong Kong
\\
\email{jiaming.ma@connect.polyu.hk} 
\and
Defeng Sun,  Corresponding Author
\at
Department of Applied Mathematics, The Hong Kong Polytechnic University, Hung Hom, Hong Kong
\\
\email{defeng.sun@polyu.edu.hk}  
}

\titlerunning{
$p$-order Semismoothness of PSD Cone Slices
}

\authorrunning{
Ruoning Chen, Jiaming Ma, Defeng Sun
}

\date{Received: date / Accepted: date}

\begin{document}

\maketitle

\begin{abstract}

The metric projection onto the positive semidefinite (PSD) cone is strongly semismooth, a property that guarantees local quadratic convergence for many powerful algorithms in semidefinite programming. In this paper, we investigate whether this essential property holds for the metric projection onto an affine slice of the PSD cone, which is the operator implicitly used by many algorithms that handle linear constraints directly. Although this property is known to be preserved for the second-order cone, we conclusively demonstrate that this is not the case for the PSD cone. Specifically, we provide a constructive example  that for any $p > 0$, there exists an affine slice of a PSD cone for which the metric projection operator fails to be $p$-order semismooth. This finding establishes a fundamental difference between the geometry of the second-order cone and the PSD cone and necessitates new approaches for both analysis and algorithm design for linear semidefinite programming problems. 



\end{abstract}


\section{Introduction}

Semidefinite programming (SDP) has emerged as a powerful framework for modeling a wide range of problems across engineering, control theory, and combinatorial optimization \cite{VB96}. A central challenge in the field is the development of robust and efficient numerical methods capable of solving these problems with high precision. For this purpose, Newton-type methods are particularly desirable because of their characteristically rapid local convergence. However, key operators in SDP, such as the metric projection onto the feasible set \cite{Za71}, are nonsmooth, which prevents the direct application of classical Newton-based approaches. This challenge has spurred the development of a broader analytical framework based on generalized derivatives, giving rise to the powerful class of nonsmooth  Newton methods \cite{Kummer88,QS93}.

The development of semismooth Newton methods, in particular, has demonstrated that establishing the semismoothness of operators involved can lead to proofs of local superlinear or even quadratic convergence \cite{QS93}. For a given positive integer $n$, let $S^n_+$ be the  cone of $n$ by $n$ real symmetric and positive semidefinite (PSD) matrices.    A PSD cone refers to   $S^n_+$  for some integer $n\ge 1$. 
In the context of SDP, a foundational result is that the metric projection operator onto $S^n_+$, denoted by $\Pi_{S^n_+}$, is strongly semismooth \cite{SS02}. This favorable property has become the basis for many solvers   such as the  Newton-CG  augmented Lagrangian methods based solvers for solving large-scale SDPs \cite{YST15,ZST10}.  

However, the feasible set of a standard SDP problem is not the entire PSD cone, but rather a slice of it, i.e., the intersection of the PSD cone with an affine subspace $L$. The practical performance of many algorithms therefore depends not on $\Pi_{S^n_+}$ but on the metric projection operator  onto this constrained set $\Pi_{S^n_+ \cap L}$. This naturally leads to a central and consequential question: Does the strong semismoothness of the metric projection onto the PSD cone carry over to the metric projection onto its affine slices?

An affirmative answer to this question would have significant implications, suggesting that a broad class of algorithms enjoys local quadratic convergence when applied directly to linear SDPs, without moving the affine constraint $A(x)=b$ from the feasible set into a penalty term in the objective function. Support for such a conjecture can be drawn from the analogous situation for the second-order cone (SOC). It is a well-established result that the metric  projection onto the SOC is strongly semismooth \cite{ChenSS2003}. Crucially, this property has been shown to be preserved for its affine slices; that is, the metric projection onto the intersection of an SOC and an affine subspace also remains strongly semismooth \cite{WX11}. Therefore, it is natural to conjecture that an analogous result may hold in the SDP setting.

However, the geometric structure of the PSD cone is known to be substantially more complex than that of the SOC. The facial structure is more intricate, and the boundary exhibits a higher degree of algebraic complexity, which often complicates the extension of results from the SOC setting \cite{Pataki00}. On the other hand, the metric projection onto a slice of PSD cone can be written in a constraint convex quadratic SDP problem. When the constraint nondegeneracy holds at the desired point $y$, the KKT solution mapping $(x(\cdot),\lambda(\cdot))$ is strongly semismooth at $ y$ \cite{Sun06,S01,MengSS05} and thus,   the metric projection part $x(\cdot)$ is strongly semismooth at $  y$. Apart from this ideal nondegenerate case, what happens outside is far less understood. 

In this paper, we provide a definitive but negative answer to this question. We demonstrate that, unlike the second-order cone case, the metric projection onto an affine slice of the PSD cone can suffer a dramatic loss of strong semismoothness. Specifically, for any $p>0$, we construct an explicit example where the metric projection operator $\Pi_{S^n_+ \cap L}$ fails to be $p$-order semismooth. This finding reveals a fundamental geometric disparity between the second-order cone and the positive semidefinite cone  and implies that the analytical foundation for rapid superlinear convergence (when $p$ near 1) of standard semismooth Newton methods does not hold universally for linear SDPs. Our result therefore clarifies a key theoretical boundary in conic optimization and underscores the need for more nuanced approaches to both the analysis and design of algorithms for this subclass of semidefinite programs.

\section{Notation and preliminaries}
Let $\cE$ and $\cF$ be two finite-dimensional real Hilbert spaces each endowed with an inner product $\langle \cdot,\cdot\rangle$ and its induced norm $\|\cdot\|$.
For a vector $z\in\cE$ (or a subspace $\cE_0\subseteq\cE$), $z^\perp$ (or $\cE_0^\perp$) denotes its orthogonal complement in $\cE$.
Given a cone $C\subseteq\cE$,
$C^\circ := \{ v \in \cE \mid \langle v, z \rangle\le 0 \  \forall \, z\in C \}$ is the polar cone of $C$. 
For a linear operator $\cA:\cE\to\cF$, we use $\cA^*, \range (\cA)$ and $\ker(\cA)$ to denote its adjoint, range space, and null space, respectively. Note that $\range (\cA)=(\ker (\cA^*))^\perp$. Given a set $C\subseteq\cE$, we use $\boundary(C)$ to denote the strong topological boundary of $C$.

Given a matrix $A\in\Re^{l\times q}$, we use $A_{ik}$ or $A_{i,k}$ to denote the entry at the $i$-th row and the $k$-th column of $A$.
Given two matrices $A,B\in\Re^{l\times q}$, the inner product of $A,B$ is defined by $\langle A, B\rangle=\trace(A^\top B)$, and $\|A\|=\sqrt{\langle A, A\rangle}$ is the Frobenius norm. The set of all symmetric matrices in $\Re^{l\times l}$ is denoted by $S^l$. The direct sum of two matrices $A$ and $B$, denoted as $A\oplus B$, is the block diagonal matrix defined as:
$$
A\oplus B:=
\begin{pmatrix}
    A & 0 \\
    0 & B
\end{pmatrix}.
$$

Let $C\subset \cE$ be a closed convex set. For a given $x\in \cE$, its metric projection onto $C$, denoted by $\Pi_C(x)$, is the unique    point in $C$  that has the minimum distance from  $x$. It is well known \cite{Za71} that for  $x\in \cE$, it holds that 
\[
\bar x := \Pi_C(x) \iff \langle x-\bar x, u-\bar x\rangle \le 0 \ \forall\, u\in C. \]
The normal cone of $C$ at $\bar z$ \cite{rocconv} is defined by
$$
N_C(\bar z)=
\begin{cases}
\{v\in\cE \mid \, \langle v, z-\bar z\rangle\le 0 
\ \forall z\in C\},
&\mbox{if } \bar z\in C,
\\
\emptyset,
&\mbox{otherwise}.
\end{cases}
$$
Note that $N_{C^{\circ}}(z)=(C^{\circ})^{\circ}\cap z^{\perp}=C\cap z^{\perp}$ when $C$ is also a cone.
To facilitate the analysis in Section 3, we state the following well-known H\"older's inequality and its equality condition. 

\begin{lemma}[H\"older's Inequality {\cite{Hardy1952}}]
\label{lem:holder}
Let $p, q \in (1, \infty)$ satisfy $\frac{1}{p}+\frac{1}{q}=1$. For any vectors $x,y\in \Re^n$, it holds that
\begin{equation*}
\sum_{i=1}^{n} |x_i y_i| \le \left( \sum_{i=1}^{n} |x_i|^p \right)^{1/p} \left( \sum_{i=1}^{n} |y_i|^q \right)^{1/q}, 
\end{equation*}
where the equality holds if and only if there exists a constant $c \ge 0$ such that $|x_i|^p = c|y_i|^q$ for all $i \in \{1, \dots, n\}$.
\end{lemma}

\subsection{Generalized differentials}
Suppose that a function $f:\cE\to\cF$ is locally Lipschitz continuous around some point $\bar x\in \cE$. Then, by Rademacher's theorem \cite[Theorem 9.60]{varbook}, $f$ is almost everywhere differentiable in a neighborhood $\cV$ of $\bar x$. We use $D_f\subseteq \cV$ to denote the set of points at which $f$ is differentiable.
The Bouligand subdifferential of $f$ at $\bar x$ is defined by 
$$
\partial_B f(\bar x):=\{ v\in \Re^{m\times n} \mid \exists \,x_k\stackrel{D_f}\longrightarrow \bar x \text{ with } f'(x_k)\to v\}. 
$$
Moreover, the Clarke generalized Jacobian of $f$ at $\bar x$ is the convex hull of $\partial_B f(\bar x)$, i.e.,  
$$
{\partial}f(\bar x):=\conv(\partial_B f(\bar x)).
$$
With this preparation, we can present the definition of $p$-order semismoothness \cite{QS93}.

\begin{definition}\label{def:ssn}
Let $f:\cE\to\cF$ be a locally Lipschitz continuous function, and let $\partial F(x)$ denote the Clarke generalized Jacobian of $F$ at $x \in \cE$.  
Let $p>0$.  We say that $F$ is \emph{$p$-order semismooth} at $x$ if:
\begin{enumerate}
    \item $f$ is directionally differentiable at $x$.
    \item for any $V \in \partial F(x+h)$, as $h \to 0$,
    \begin{equation}\label{eq:ssn1}
        \| f(x+h) - f(x) - Vh \| = \cO(\|h\|^{1+p}).
    \end{equation}
\end{enumerate}
We say that the function $f$ is \emph{strongly  semismooth} at $x$ if it is  $1$-order semismooth  at $x$.
\end{definition}

According to \cite[Theorem 2.3]{QS93}, we know that \eqref{eq:ssn1} in Definition \ref{def:ssn} can be replaced by
\begin{equation}
    \label{eq:ssn2}
    f(x+h)-f(x)-f'(x+h;h)=\cO(\|h\|^{1+p}),
\end{equation}
where $f'(x+h;h)$ means the directional derivative of $f$ at $x+h$ along the direction $h$. Note that a function $\phi(t)$ is \emph{$\cO(t)$} if there exists a constant $c$ such that $\|\phi(t)\|\le c|t|$ for all sufficiently small $|t|$; it is \emph{$\Theta(t)$} if there exist constants $c_1,c_2>0$ such that $c_1|t| \le \|\phi(t)\| \le c_2|t|$ for all sufficiently small $|t|$.

\subsection{LMI-representable sets}
Following \cite[Equation (1.3)]{HN10}, we introduce the definition of LMI-representable set, which is the core tool of this research.

\begin{definition}[LMI‑representable set/Spectrahedron]
    A set $K\subseteq \Re^n$ is said to be LMI (Linear Matrix Inequality) representable or is a spectrahedron if there exist $m>0$ and $A_0,\dots,A_n\in S^m$ such that
    \begin{equation*}
        K=\{(x_1,\dots,x_n)\in \Re^n \mid A_0+A_1x_1+\dots+A_nx_n\in S^m_+\}.
    \end{equation*}
\end{definition}
LMI-representable sets are closely related to the slices of PSD cone. To illustrate this, consider the spectrahedron $K_1=\{x\in \Re^n\mid \cA(x)+A_0\in S^m_+\}$ for a linear operator $\cA: \Re^n \to S^m$. If $\cA$ is injective, then $K_1$ is linearly isomorphic to the set $K_2=S^m_+\cap \{z\in S^m\mid \cB z=b\}$ for some linear operator $\cB$ and vector $b$. The set $K_2$ is a slice of $S^m_+$. We have already known that the metric projection onto the PSD cone $\Pi_{S^m_+}$ is strongly semismooth \cite[Corollary 4.15]{SS02}. It is of our interest to study the strong semismoothness of $\Pi_{K_2}$, the metric projection onto a slice of PSD cone. This operator $\Pi_{K_2}$ is apparently a semi-algebraic Lipschitz function. Thus, for any fixed $\bar x$ we have a rational number $\gamma_{\bar x}>0$ such that  $\Pi_{K_2}(\cdot)$ is $\gamma_{\bar x}$-order semismooth at $\bar x$ \cite[Proposition 1]{BDL07}. Therefore, it is of curiosity for a given $p>0$, whether there exist $m>0,\cB$ and $b$ such that $\Pi_{K_2}$ fails to be $p$-order semismooth at some point. To answer this, we first study the semismoothness of $\Pi_{K_1}$ for an injective $\cA$, and take advantage of the linear isomorphism between $K_1$ and $K_2$. For this purpose, we first introduce the following lemma.

\begin{lemma}\label{lem:linear-iso}
    Let $\mathcal{A}:\cE\to\cF$ be an injective linear operator, and $K\subseteq\cE$ be a closed convex set. For any given $p>0$, the $p$-order semismoothness of $\Pi_K$ is equivalent to the $p$-order semismoothness of $\Pi_{\mathcal{A}K}$.
\end{lemma}
\begin{proof} 
Since \(\mathcal{A}\) is an injective linear operator between finite-dimensional spaces and \(K\) is a closed convex set, its image \(\mathcal{A}K\) is also closed and convex. This ensures that the metric projection \(\Pi_{\mathcal{A}K}\) is well-defined.

Next, we show that the $p$-order semismoothness of $\Pi_{\cA K}$ on $\cF$ is equivalent to that of its restriction to $\range(\cA)$. Since $\cA K \subseteq \range(\cA)$, we have $\Pi_{\cA K}(x) = \Pi_{\cA K}(\cB x)$ for any $x\in\cF$, where $\cB:=\Pi_{\range{\cA}}$ is a linear operator. From the definition of the directional derivative, we know that  for any $x,h\in\cF$,
\[\Pi'_{\cA K}(x; h) = \Pi'_{\cA K}(\cB x; \cB h)\]  and 
\begin{equation*}
\begin{aligned}
    &\Pi_{\cA K}(x+h)-\Pi_{\cA K}(x)-\Pi'_{\cA K}(x;h)\\
    = &\Pi_{\cA K}(\cB x+\cB h)-\Pi_{\cA K}(\cB x)-\Pi'_{\cA K}(\cB x; \cB h),
\end{aligned}
\end{equation*}
which demonstrates that the expression depends only on the components within $\range(\cA)$. Therefore, $\Pi_{\cA K}$ is $p$-order semismooth on $\cF$ if and only if its restriction to $\range(\cA)$ is $p$-order semismooth. This allows us to assume, without loss of generality, that $\range(\cA) = \cF$, which implies that $\cA$ is an invertible linear operator.

For $x \in \cF$, let $y = \Pi_{\cA K}(x)$. By definition, $y$ is the unique point in $\cA K$ satisfying the following  inequality:
\[ \langle x-y,c-y \rangle\le 0 \quad \forall \,c \in \cA K. \]
Since $y \in \cA K$ and $\cA$ is a bijection, there exists a unique $z \in K$ such that $y = \cA z$. Similarly, any $c \in \cA K$ can be written as $c=\cA w$ for some $w \in K$. Substituting these into the above  inequality gives:
\[ \langle x - \cA z,\cA w - \cA z\rangle \le 0 \quad \forall \,w \in K. \]
Using the linearity of $\cA$ and the property of the adjoint operator $\cA^*$, we obtain:
\[ \langle \cA^*\cA z - \cA^*x,w-z\rangle \ge 0 \quad \forall w \in K, \]
which is the first-order optimality condition of the strongly convex program 
$$\min_{u \in K} \frac{1}{2}\langle \cA^*\cA u,u\rangle - \langle \cA^*x,u\rangle,$$
whose unique solution is denoted by $z(x)$, and we have $\Pi_{\cA K}(x) = \cA(z(x))$.

Let $0<\lambda_\text{min}\leq\dots\leq\lambda_\text{max}$ be the eigenvalues of $\cA^*\cA$ and $0<\gamma<1/\lambda_\text{max}$. Consider the function $F: \cE \times \cF \to \cE$:
\[ F(z, x) := z - \Pi_K\left[z - \gamma(\cA^*\cA z - \cA^*x)\right]. \]
It is clear that $F(z(x),x)=0$ for any $x\in\cF$. First suppose that $\Pi_K$ is $p$-order semismooth. Since $F$ is a composition of affine operators and $\Pi_K$, it is $p$-order semismooth with respect to the joint variable $(z,x)$. Let $\pi_z \partial F(z,x)$ be the canonical projection of $\partial F(z,x)$ onto the $z$-part. Since $I -\gamma \cA^* \cA$ is invertible, an element $V \in \pi_z \partial F(z,x)$   has the form $V = I - W (I - \gamma \cA^*\cA)$, where $W\in\partial\Pi_K[z - \gamma(\cA^*\cA z - \cA^*x)]$. Since $\Pi_K$ is non-expansive, we have $\|W\| \le 1$. Thus, we have
\[ \|W(I - \gamma \cA^*\cA)\| \le \|W\|\|I - \gamma \cA^*\cA\| \le \max(|1-\gamma\lambda_\text{min}|, |1-\gamma\lambda_\text{max}|) < 1, \]
which implies that $V$ is invertible.  

By noting that the implicit function theorem for the (strongly) semismooth function \cite[Theorem 2.1]{S01} can easily be extended to the $p$-order case, as pointed out at the end of its proof, we know that $F$ satisfies the conditions of the implicit function theorem and $F(z(x),x)=0$ for any $x\in\cF$. Therefore, $z(x)$ is $p$-order semismooth. Recall that $\Pi_{\cA K}(x) = \cA(z(x))$. We have that $\Pi_{\cA K}$ is $p$-order semismooth. Since $\cA$ is invertible, the $p$-order semismoothness of $\Pi_{\cA K}$ also implies the $p$-order semismoothness of $\Pi_K$. This completes the proof.\qed
\end{proof}

\section{A \texorpdfstring{$p$}{p}-order semismooth example}
In this section, we first present that for any $p>0$, there exists an LMI-representable set $K$ such that $\Pi_K$ fails to be $p$-order semismooth. Then combined with Lemma \ref{lem:linear-iso}, we extend this result to the slices of the PSD cone.

Let $n\geq 2$ be an integer. Consider the closed convex cone
\begin{equation*}
    \begin{aligned}
        K_n:=\bigg\{(x_1,x_2,x_3,y_1,\dots,y_{n-1},z_1,\dots,z_{n-1})\mid x_3^2\ge y_1^2+z_1^2, \;x_3\ge 0&,\\
        x_3y_i\ge y_{i+1}^2 \text{ for } i=1,\dots, n-2,\quad x_3y_{n-1}\ge x_1^2&,\\
        x_3z_i\ge z_{i+1}^2 \text{ for } i=1,\dots, n-2,\quad x_3z_{n-1}\ge x_2^2&\bigg\}.
    \end{aligned}
\end{equation*}
 Define the following linear operators from $\mathbb{R}^{2n+1}$ to $S^2$:
\begin{equation*}
    \begin{aligned}
        &M_0(p)=
        \begin{pmatrix}
            x_3+y_1&z_1\\
            z_1&x_3-y_1
        \end{pmatrix},\\
        &M_i^y(p)=
        \begin{pmatrix}
            x_3&y_{i+1}\\
            y_{i+1}&y_i
        \end{pmatrix} \text{for } i=1,\dots, n-2,\quad 
        M_{n-1}^y(p)=
        \begin{pmatrix}
            x_3&y_{n-1}\\
            y_{n-1}&x_1
        \end{pmatrix},\\
        &M_i^z(p)=
        \begin{pmatrix}
            x_3&z_{i+1}\\
            z_{i+1}&z_i
        \end{pmatrix} \text{for } i=1,\dots, n-2,\quad 
        M_{n-1}^z(p)=
        \begin{pmatrix}
            x_3&z_{n-1}\\
            z_{n-1}&x_2
        \end{pmatrix}, 
    \end{aligned}
\end{equation*}
where $p:=(x_1,x_2,x_3,y_1,\dots,y_{n-1},z_1,\dots,z_{n-1})\in\mathbb{R}^{2n+1}$. It is easy to check that $K_n$ can be written in the following form:
\begin{equation}
    \begin{aligned}\label{Axform}
        K_n=\Bigg\{&p\in\mathbb{R}^{2n+1} \mid M_0(p) \oplus \Bigg(\bigoplus_{i=1}^{n-1} M_i^y(p)\Bigg) \oplus \Bigg(\bigoplus_{i=1}^{n-1} M_i^z(p)\Bigg) \in S_+^{4n-2}
        \Bigg\}.
    \end{aligned}
\end{equation}
One can see that $K_n$ is an LMI-representable set. For notational simplicity, we denote $\kappa_n:=2^n$. Consider the closed convex cone $S_n$ that represents the projection of $K_n$ onto the space of the first three variables, i.e.,
\begin{equation}\label{link}
    \begin{aligned}
        S_n=&\{(x_1,x_2,x_3)\mid x_1^{\kappa_n}+x_2^{\kappa_n}\le x_3^{\kappa_n},\;x_3\ge 0\} \\
        =&\{(x_1,x_2,x_3)\mid\exists y_i,z_i \text{ for } i=1,\dots,n-1, \\
        &\text{ s.t. } (x_1,x_2,x_3,y_1,\dots,y_{n-1},z_1,\dots,z_{n-1})\in K_n\}.
    \end{aligned}
\end{equation}
Let $\lambda_n:=2^n/(2^n-1)$. Since $\kappa_n$-norm and $\lambda_n$-norm are dual norms, we have
\begin{equation*}
    S_n^\circ=\{(u_1,u_2,u_3)\mid u_1^{\lambda_n}+u_2^{\lambda_n}\le u_3^{\lambda_n},\;u_3\le 0\}.
\end{equation*}
With \eqref{link}, it is easy to know
\begin{equation*}
    S_n^\circ=\{(u_1,u_2,u_3)\mid (u_1,u_2,u_3,0,\dots,0)\in K_n^\circ\}.
\end{equation*}
We have the following lemma about $N_{K_n^{\circ}}$.
\begin{lemma}\label{lem:normal}
    Let $u_1,u_2\geq0$ and $v:=(u_1,u_2,-1,0,\dots,0)\in \boundary(K_n^\circ)$.  Then
    \begin{equation*}
        N_{K_n^{\circ}}(v)=\cone(w),
    \end{equation*}
    where $w:=(u_1^{\lambda_n/2^n},u_2^{\lambda_n/2^n},1,u_1^{\lambda_n/2^1},\dots,u_1^{\lambda_n/2^{n-1}},u_2^{\lambda_n/2^1},\dots,u_2^{\lambda_n/2^{n-1}})$.
\end{lemma}
\begin{proof}
    Since $K_n$ and $K_n^{\circ}$ are closed convex cones, we have
    \begin{equation*}
        N_{K_n^{\circ}}(v)=(K_n^\circ)^\circ\cap v^\perp=K_n \,\cap \, v^{\perp}.
    \end{equation*}
    Let $(x_1,x_2,x_3,y_1,\dots,y_{n-1},z_1,\dots,z_{n-1})\in N_{K_n^{\circ}}(v)$. If $x_3=0$, by the definition of $K_n$, we have $x_1=x_2=0$, $y_i=0$ and $z_i=0$. So, we can assume that  $x_3=1$ since $x_3\ge0$. Then  
    \begin{equation*}
        x_1u_1+x_2u_2=1.
    \end{equation*}
    Note that $1/\kappa_n+1/\lambda_n=1$.  By H\"older's inequality in Lemma \ref{lem:holder}, we have
    \begin{equation*}
        x_1u_1+x_2u_2\le(x_1^{\kappa_n}+x_2^{\kappa_n})^{1/\kappa_n}(u_1^{\lambda_n}+u_2^{\lambda_n})^{1/\lambda_n}\le 1.
    \end{equation*}
    {}From  the equality condition of H\"older's inequality in Lemma \ref{lem:holder}, we have $x_1^{\kappa_n}=u_1^{\lambda_n}$ and $x_2^{\kappa_n}=u_2^{\lambda_n}$. So $x_1=u_1^{\lambda_n/2^n}$ and $x_2=u_2^{\lambda_n/2^n}$ since $u_1,u_2\ge0$ and $x_1u_1+x_2u_2=1$. By using the definition of $K_n$ and noting that $x_1^{\kappa_n}+x_2^{\kappa_n}=1$, we have $y_i=u_1^{\lambda_n/2^i}$ and $z_i=u_2^{\lambda_n/2^i}$. This completes the proof. \qed
\end{proof}

Now, we are ready to state our    main result as in the following theorem.
\begin{theorem}\label{thm:p-semismooth}
    For any $p>0$, there exists an LMI-representable set $K$ such that $\Pi_K$ fails to be $p$-order semismooth.
\end{theorem}
\begin{proof}
 Consider the following curve:
\begin{equation}\label{formula_v}
    v(t):=(t,(1-t^{\lambda_n})^{1/\lambda_n},-1,0,\dots,0)\in  K_n^\circ,\quad t\in[0,1].
\end{equation}
Let
\begin{equation}\label{formula_w}
    w(t):=(t^{\lambda_n/\kappa_n},(1-t^{\lambda_n})^{1/\kappa_n},1,y_1(t),\dots y_{n-1}(t),z_1(t),\dots z_{n-1}(t)),
\end{equation}
where $y_i(t)=t^{\lambda_n/2^i}$ and $z_i(t)=(1-t^{\lambda_n})^{1/2^i}$. So $N_{K_n^\circ}(v(t))=\cone(w(t))$ by Lemma \ref{lem:normal}. 
Letting $t\to0$, we evaluate the semismoothness as follows:
\begin{equation*}
    \begin{aligned}
        &v(t)-v(0)-\Pi'_{K_n^{\circ}}(v(t);v(t)-v(0))\\
        \overset{(a)}{=}&v(t)-v(0)-\Pi_{T_{K_n^{\circ}}(v(t))}(v(t)-v(0))\\
        =&v(t)-v(0)-\Big(v(t)-v(0)-\frac{\langle v(t)-v(0),w(t)\rangle}{\|w(t)\|^2}w(t)\Big)\\
        =&\frac{\langle v(t)-v(0),w(t)\rangle}{\|w(t)\|^2}w(t)\\
        \overset{(b)}{=}&\Theta(t^{\lambda_n}+{(1-t^{\lambda_n})^{1/\kappa_n}}-1)w(t)\\
        \overset{(c)}{=}&\Theta(t^{\lambda_n}),
    \end{aligned}
\end{equation*}
where $(a)$ follows from \cite[Equation (1.2)]{Za71}, $(b)$ follows from \eqref{formula_v} and \eqref{formula_w}, and $(c)$ comes from $(1+x)^a-1=\Theta(x)$ when $x\to0$ and $a>0$. Note that $\|v(t)-v(0)\|=\Theta(t)$, we have that $\Pi_{K_n^{\circ}}$ is at most $(\lambda_n-1)$-order semismooth from \eqref{eq:ssn2}. 
So, $\Pi_{K_n}$ is also at most $(\lambda_n-1)$-order semismooth. Therefore, Theorem \ref{thm:p-semismooth} holds following by $\lambda_n\to 1$ as $n\rightarrow\infty$.
\qed \end{proof}

Recall that $K_n$ in \eqref{Axform} can be written in the form $\{x\in \mathbb{R}^{2n+1}\mid \mathcal{A}x\in S_+^{4n-2}\}$ where $\mathcal{A}:\mathbb{R}^{2n+1}\to S^{4n-2}$ is a linear operator. It is directly to check that $\mathcal{A}$ is injective. Let $T_n:=\mathcal{A}K_n\subset S_+^{4n-2}$. So $T_n=S_+^{4n-2}\cap \range(\mathcal{A})$. Thus, there exists a linear operator $\mathcal{B}:S^{4n-2}\to\mathbb{R}^m$ such that $T_n=\{X\in S_+^{4n-2}\mid \mathcal{B}X=0\}$. So, by Lemma \ref{lem:linear-iso}, we have that $\Pi_{T_n}$ is at most $(\lambda_n-1)$-order semismooth.

Note that $T_n=S_+^{4n-2}\cap\ker(\mathcal{B})$ is the intersection of the strongly semismooth closed convex cone \cite{SS02} and a linear subspace. Therefore, we have 
\begin{equation}
    \label{sum_form}
    \begin{aligned}
        T_n^{\circ}\overset{(a)}{=}&S_-^{4n-2}+\ker(\mathcal{B})^{\circ}\\
        \overset{(b)}{=}&S_-^{4n-2}+\range(\mathcal{B}^*),
    \end{aligned}
\end{equation}
where $(a)$ comes from \cite[Theorem 6.42]{varbook} and $(b)$ comes from $(\ker(\mathcal{B}))^{\perp}=\range(\mathcal{B}^*)$. From \eqref{sum_form}, $T_n^{\circ}$ is the Minkowski sum of a subspace in $S^{4n-2}$ and a strongly semismooth convex closed cone. However, $\Pi_{T_n^{\circ}}$ shares the same semismooth order with $\Pi_{T_n}$, and thus is at most $(\lambda_n-1)$-order semismooth. Therefore, we have the following corollaries.

\begin{corollary}
    For any $p>0$, there exists a slice $K$ of PSD cone such that $\Pi_K$ fails to be $p$-order semismooth.
\end{corollary}

\begin{corollary}
    For any $p>0$, there exists a set $K$ such that $K$ is the Minkowski sum of a nonpolyhedral strongly semismooth set and a polyhedral convex set but  $\Pi_K$ fails to be $p$-order semismooth.
\end{corollary}

\section{Conclusion}
In this paper, we have demonstrated that the metric  projection onto a slice of the PSD cone may fail to be $p$-semismooth for any $p>0$. Building on this finding, we conclude our paper by raising the following two natural open questions:

\textbf{Question 1.} According to \cite{SS08}, the metric projection onto a symmetric cone (i.e., a self-dual homogeneous cone) is strongly semismooth. Moreover, every homogeneous cone or symmetric cone is linearly isomorphic to a slice of the PSD cone \cite{C04,Faybusovich02}. By Lemma~\ref{lem:linear-iso}, this implies that there exist many special slices of the PSD cone for which the metric projection is strongly semismooth. A fundamental open problem is to identify sufficient (and necessary) conditions under which the metric projection onto a slice of the PSD cone is strongly semismooth. For instance, is the metric projection onto an arbitrary homogeneous cone always strongly semismooth?

\textbf{Question 2.} In our constructions, we need a higher-dimensional PSD cone to get lower-order semismoothness. On the other hand, the metric projection onto any slice of $S^2_+$ (which is linearly isomorphic to an SOC) must be strongly semismooth. This raises the possibility that the order of semismoothness is intrinsically linked to the dimension of the underlying PSD cone. An important question is therefore to establish the lower bound on the attainable order of semismoothness as a function of the dimension of the PSD cone.

\section*{Declarations}

{\bf Conflict of interest}
The authors have no relevant financial or non-financial interests to disclose.



\begin{thebibliography}{10}

\bibitem{BDL07}
Bolte, J., Daniilidis, A., Lewis, A.:
Tame functions are semismooth. 
Math. Program. \textbf{117}, 5--19 (2009)

\bibitem{ChenSS2003}
Chen, X.D., Sun, D.F., Sun, J.: 
Complementarity functions and numerical experiments for second-order-cone complementarity problems. 
Comput. Optim. Appl. \textbf{25}, 39--56 (2003)  

\bibitem{C04}
Chua, C. B.:
Relating homogeneous cones and positive definite cones via T-algebras.
SIAM J. Optim. \textbf{14}, 500--506 (2004)



\bibitem{Faybusovich02} Faybusovich, L.:  On Nesterov's approach to semi-definite programming. 
Acta Appl. Math. \textbf{74}, 195--215 (2002).



\bibitem{Hardy1952}
Hardy, G.H., Littlewood, J.E., P\'{o}lya, G.:
Inequalities, 2nd ed., Cambridge University Press, Cambridge (1952)



\bibitem{HN10}
Helton J.W., Nie J.W.:
Semidefinite representation of convex sets. 
Math. Program. \textbf{122}, 21--64 (2010)






\bibitem{Kummer88}
Kummer, B.:   
Newton’s method for non-differentiable functions. In: Guddat J et al., eds. Advances in Mathematical Optimization (De Gruyter,
Berlin), pp. 114--125 (1988). 

\bibitem{MengSS05}
Meng, F.W., Sun, D.F., Zhao, G.Y.: 
Semismoothness of solutions to generalized equations and the Moreau-Yosida regularization. 
Math. Program. \textbf{104}, 561--581 (2005)


\bibitem{Pataki00}
Pataki, G.: 
The geometry of semidefinite programming. 
In: Wolkowicz, H., Saigal, R., Vandenberghe, L. (eds.) Handbook of Semidefinite Programming: Theory, Algorithms, and Applications, pp. 29--65. Springer, Boston (2000).

\bibitem{QS93}
Qi, L.Q., Sun, J.: 
A nonsmooth version of Newton's method.
Math. Program. \textbf{58}, 353--367 (1993)



\bibitem{rocconv}
Rockafellar, R.T.: 
Convex Analysis. 
Princeton University Press, Princeton, New Jersey (1970)

\bibitem{varbook}
Rockafellar, R.T., Wets, R.J.-B.: 
Variational Analysis.
Springer-Verlag, New York (1998)

\bibitem{S01}
Sun, D.F.:
A further result on an implicit function theorem for locally Lipschitz functions.
Oper. Res. Lett. \textbf{28}, 193--198 (2001)

\bibitem{Sun06}
Sun, D.F.: The strong second order sufficient condition and constraint nondegeneracy in nonlinear semidefinite programming and their implications. 
Math. Oper. Res. \textbf{31}, 761--776 (2006)



\bibitem{SS02}
Sun, D.F., Sun, J.:
Semismooth matrix-valued functions. 
Math. Oper. Res. \textbf{27}, 150--169 (2002)

\bibitem{SS08}
Sun, D.F., Sun, J.:
L{\"o}wner's operator and spectral functions in Euclidean Jordan algebras. 
Math. Oper. Res. \textbf{33}, 421--445 (2008)

\bibitem{VB96}
Vandenberghe, L., Boyd, S.: 
Semidefinite programming. 
SIAM Rev. \textbf{38}, 49--95 (1996) 

\bibitem{WX11}
Wang, Y.N., Xiu, N.H.:
Strong semismoothness of projection onto slices of second-order cone.
J. Optim. Theory Appl. \textbf{50}, 599--614 (2011)

\bibitem{YST15}
Yang, L.Q., Sun,  D.F., Toh, K.-C.: 
SDPNAL+: a majorized semismooth Newton-CG augmented Lagrangian method for semidefinite programming with nonnegative constraints. 
Math. Program. Comput. \textbf{7}, 331--366 (2015)

\bibitem{Za71}
Zarantonello, E.H.:
Projections on convex sets in Hilbert space and spectral theory, In contributions to Nonlinear Functional Analysis, Academic Press, New York, 237--424 (1971)

\bibitem{ZST10}
Zhao, X.Y., Sun, D.F., Toh, K.-C.:
A Newton-CG augmented Lagrangian method for semidefinite programming. 
SIAM J. Optim. \textbf{20}, 1737--1765 (2010)


\end{thebibliography}
\end{document}